\documentclass[10pt]{amsart}
\usepackage{amsmath,mathrsfs,amssymb}
\usepackage{fullpage}

\newtheorem{teo}{Theorem}
\newtheorem{pro}{Proposition}
\newtheorem{lem}{Lemma}
\newtheorem{cor}{Corollary}
\newtheorem*{rem}{Remark}

\title{Moments of the free Jacobi process: a matrix approach}

\author[N. Demni]{Nizar Demni}
\address{Aix-Marseille Univ., CNRS, I2M, UMR 7373, 39, rue
Fr\'ed\'eric Joliot Curie, 13453 Marseille Cedex 13, France}
\email{nizar.demni@univ-amu.fr}

\author[T. Hamdi]{Tarek Hamdi}
\address{Department of Management Information Systems  \\ College of Business and Economics\\ Qassim University  \\ Saudi Arabia
and Laboratoire d'Analyse Math\'ematiques et applications LR11ES11 \\ Universit\'e de Tunis El-Manar \\ Tunisie}
\email{t.hamdi@qu.edu.sa}

\usepackage{graphicx}

\begin{document}

\maketitle

\begin{abstract}
We compute the large size limit of the moment formula derived in \cite{DHS} for the Hermitian Jacobi process at fixed time. Our computations rely on the polynomial division algorithm which allows to obtain cancellations similar to those obtained in Lemma 3 in \cite{Bia}. 
In particular, we identify the terms contributing to the limit and show they satisfy a double recurrence relation. We also determine explicitly some of them and revisit a special case relying on Carlitz summation identity for terminating $1$-balanced ${}_4F_3$ functions taken at unity. 
\end{abstract}
\tableofcontents
\section{Introduction, reminder and main result}
\subsection{Random matrices and matrix-valued stochastic processes} 
The study of moments of random matrices and of their large size limits was initiated by Wigner and Dyson in their analysis of spectra of heavy nuclei (see \cite{For}). For unitarily-invariant self-adjoint and unitary matrix models, these limits fit the moments of spectral distributions of self-adjoint and unitary operators in a non commutative probability space endowed with a trace (see e.g. \cite{Hia-Pet}). Famous and widely studied examples include Wigner semi-circle, Marchenko-Pastur, Wachter and Haar distributions.    
The third one, named after K. W. Wachter who introduced it more than four decades ago, includes up to a variable change the so-called Kesten-McKay distribution arising in the study of symmetric random walks on free groups (\cite{Kes}) and of large random regular graphs (\cite{McKay}). It also arises as the limiting empirical distribution of matrices drawn from the Jacobi unitary ensemble (\cite{Cap-Cas}, \cite{Col}) and an elegant expression of its moments may be found in \cite{CDLV}.    

Dynamical versions of random matrices are, when they exist, matrix-valued stochastic processes. For instance, Gaussian ensembles and Wishart matrices correspond to marginals (fixed time) matrix Brownian motions and Wishart processes respectively. While matrices drawn from the classical compact Lie groups, known also as Dyson circular models, show up in the limit as time goes to infinity of Brownian motions on the aforesaid groups. In particular, the limiting empirical distribution of the Brownian motion on the unitary group was determined in \cite{Bia} relying on the expansion of the heat kernel in the Schur polynomial basis. Subsequently, traces of powers of Brownian motions in classical Lie group were thoroughly studied in \cite{Lev} and \cite{Dah} making use of algebraic-combinatorics and representation theory.  
 
\subsection{The Hermitian and free Jacobi processes} 
The dynamical version of the Jacobi unitary ensemble is the so-called complex Hermitian Jacobi process $(J_t)_{t \geq 0}$ whose marginal distribution converges weakly to it as $t \rightarrow +\infty$. This matrix-valued stochastic process is constructed as the radial part of an upper-left corner of a Brownian motion in the unitary group (see \cite{Del-Dem} and references therein). The stochastic properties of its eigenvalues process and of their $\beta$-extension were studied in \cite{Dem0}, while the low and high temperature behaviors of these particle systems were determined in \cite{Voi} and \cite{TT} respectively.
   
In the large size limit, the Hermitian Jacobi process converges (when suitably normalized) in the sense of mixed moments to the free Jacobi process introduced in \cite{Dem}. In particular, at any fixed time $t > 0$, the limiting moment sequence of the former 
uniquely determines the spectral distribution of the latter in a non commutative probability space. Denoted $\mu_t^{(\lambda, \theta)}$, this compactly-supported probability measure depends on two parameters $\lambda \in ]0,1], \theta \in ]0,1[,$ which respectively encodes in the large size limit the rectangular shape of the corner out of which $J_t$ is built and its magnitude regarding the size of the underlying Brownian motion in the unitary group (see below). To the best of our knowledge, the moments of  $\mu_t^{(\lambda, \theta)}$ are explicit only in the particular case corresponding to $\lambda = 1$ and $\theta = 1/2$. However, the Lebesgue decomposition of $\mu_t^{(\lambda, \theta)}$ was determined in \cite{Ham}. The corresponding absolutely continuous part is encoded by the inverse (in a compositional sense) of the solution to a radial Loewner equation. Specifically, this inverse extends continuously to the unit circle \cite{Ham1}. Unfortunately, in general, the explicit expressions for the moments and the density remains elusive.



Motivated by quantum information theory and especially by dynamical random quantum states, the moments of $\mu_t^{(1,1/k)}, k \geq 2,$ has been further analyzed in \cite{Dem-Ham} relying on enumeration techniques. Though this analysis reveals a quite striking combinatorial structure played by the Catalan triangle (\cite{Sha})
\begin{equation}\label{Catalan}
\binom{2n}{n-h}, \quad 0 \leq h \leq n,
\end{equation}
independently of $k$, it involves the moments of a non normal bounded operator which do not seem to have any simple form unless $k=2$. Besides, unless $h=0$, the presence of the above binomial coefficients in relation with the moments of the free Jacobi process remains till now unexplained. As to the binomial coefficient
\begin{equation*}
\binom{2n}{n},
\end{equation*}
it arises in the moment sequence of the stationary (arcsine) distribution $\mu_{\infty}^{(1, 1/2)}$ and counts the number of non crossing partitions of type B (\cite{Rei}). 
\subsection{Moments of the Hermitian Jacobi process}
In order to get more insight into both the combinatorial and analytical structures of the moment sequence of $\mu_t^{(\lambda, \theta)}$, in particular into the occurrence of the above binomial coefficients, we pursued a matrix approach in \cite{DHS} where we proved a formula for the moments $\mathbb{E}\left(\textrm{tr}[(J_{t})^n]\right), n \geq 1,$. The latter, recalled below, is given by a sign alternating double sum over hooks of fixed weight $n$ and those contained in. It may be also expressed through terminating 
and $1$-balanced ${}_4F_3$ hypergeometric functions taken at unity (\cite{Bai}). However, computing the large size limit is still challenging due to the presence of alternating signs and to the complicated form of the summands as well.

At this stage, let us provide more details on the moment formula derived in \cite{DHS}. To this end, consider a $d \times d$ Brownian motion on the unitary group and its upper-left corner of shape $m \times p$, where without loss of generality we assume 
$p \geq m$. Then $J_t$ is the $m \times m$ radial part of this corner. Now, recall from \cite{DHS} the parameters
\begin{equation*}
r:= p-m,  \quad q:=d-p, \quad s:= q-m,
\end{equation*}
and denote
\begin{equation}\label{MomSeq}
M_{n}^{(r,s,m)}(t):= \mathbb{E}\left(\textrm{tr}[(J_{t/d})^n]\right),
\end{equation}
the time-normalized $n$-th moment of $J_{t/d}$. Given a hook $\alpha =  (n-k, 1^k)$, denote further $|\alpha| = n$ its weight, $\alpha_1 = n-k$ its largest component, and $\l(\alpha) = k+1$ its length. 
Then, Theorem 1 in \cite{DHS} reads: 
\begin{equation}\label{FinalForm}
M_{n}^{(r,s,m)}(t) = M_n^{(r,s,m)}(\infty) + \sum_{\substack{\alpha \, \, \textrm{hook} \\ |\alpha| = n, l(\alpha) \leq m}} (-1)^{n-\alpha_1}  \sum_{\substack{\tau \subseteq \alpha \\ \tau \neq \emptyset}}
	\frac{e^{-\nu_\tau(r,s,m) t/d} \, {V}_{\alpha_1, \tau_1}^{r,s,m} \, U_{l(\alpha), l(\tau)}^{r,s,m} }{(r+s+\tau_1+2m-l(\tau)) (\tau_1+l(\tau)-1)}.
\end{equation}
In \eqref{FinalForm}, the notation $\tau \subseteq \alpha$ stands for the containment of the corresponding Young diagrams (so $\tau$ is a hook as well and $\tau_1$ is its largest component), 
\begin{align}\label{Eigen}
	\nu_{\tau}(r,s,m) :=\sum_{i=1}^m\tau_i(\tau_i+r+s+1+2(m-i)),
\end{align}
\begin{align}\label{V}
	{V}_{\alpha_1, \tau_1}^{(r,s,m)} := \frac{(d+2\tau_1-1)\Gamma(d+\tau_1)\Gamma(\alpha_1+m)\Gamma(p+\alpha_1)\Gamma(q+\tau_1)}{(\alpha_1-\tau_1)!\Gamma(d+\alpha_1+\tau_1)\Gamma(p+\tau_1)\Gamma(\tau_1)},
\end{align}
\begin{align}\label{U}
	U_{l(\alpha), l(\tau)}^{(r,s,m)} := \frac{(d+1-2l(\tau))\Gamma(d-l(\alpha)-l(\tau)+1)\Gamma(p-l(\tau)+1)}{(l(\alpha)- l(\tau))!(l(\tau)-1)!\Gamma(m-l(\alpha)+1)\Gamma(p-l(\alpha)+1)\Gamma(q-l(\tau)+1)\Gamma(d-l(\tau)+1)},
\end{align}
and $M_n^{(r,s,m)}(\infty)$ is the $n$-th moment of the Jacobi unitary ensemble. We refer the reader to \cite{Kra} for an expression of the moment sequence $(M_n^{(r,s,m)}(\infty))_{n \geq 0}$ and for its asymptotic analysis based on terminating ${}_4F_3$ hypergeometric functions taken at unity.  

Since any hook $\tau \subseteq \alpha, \tau \neq \emptyset,$ is a partition of the form
\begin{equation*}
\tau := (h-j,1^j), \quad 0 \leq j \leq h-1, \quad 1 \leq h \leq n,
\end{equation*}
then the double sum displayed in the RHS of \eqref{FinalForm} may be written as (\cite{DHS}, section 4.2): 
 \begin{multline}\label{HG}
\sum_{h=1}^n \sum_{j=0}^{h-1}\frac{(-1)^{n-h+j}e^{-\nu_{h,j(d)} t/d}}{h(n-h)!(h-j-1)!j!}  \frac{(2m+2h-2j-1)(d-2j-1)\Gamma(p-j)\Gamma(q+h-j)\Gamma(d+h-j)}{(d+h-2j-1)\Gamma(q-j)\Gamma(d-j)}
\\ \frac{\Gamma(h-j+m)\Gamma(d-n+h-2j-1)}{\Gamma(m+h-n-j)\Gamma(p+h-n-j)\Gamma(d+2h-2j)} 
{}_4F_3\left(\begin{matrix} -(n-h), m+h-j, p+h-j, d-n+h-2j-1; \\ m-n+h-j, p-n+h-j, d+2h-2j; \end{matrix}, 1\right),
\end{multline}
where
\begin{align}\label{Eigen1}
\nu_{h,j}(d) := \nu_{(h-j,1^j)}(r,s,m) =  dh + (h-j)(h-j-1) - j(j+1) = dh +h(h-2j-1).
\end{align}
Here, ${}_4F_3$ is a terminating and $1$-balanced (i.e. the sum of the bottom row parameters equals one plus the sum of the top row parameters) hypergeometric functions taken at unity. Various parameter transformations and reduction formulas exist 
for these functions (\cite{Bai}). For instance, generalized hypergeometric series whose row and bottom parameters differ by integers admits simpler representations (see \cite{Mil-Par} and references therein). Recently, a reduction formula proved in \cite{Chu} (Theorem 2.1) reduces a $1$-balanced terminating ${}_4F_3(1)$ to a terminating (non necessarily balanced) ${}_3F_2(1)$. In particular, it includes the celebrated Carlitz Summation formula for a special choice of the parameter set (\cite{Abi}). 
 As we shall see below, the latter allows to simplify considerably \eqref{HG} when choosing $p(m) = m+(1/2), d(m)=2m,$ which in turn opens the way to retrieve the moment sequence of $\mu_t^{(1,1/2)}$ derived in \cite{DHH}. 
Note in passing that though this choice does not correspond to the Hermitian Jacobi process with integer parameters, the eigenvalue process  still exists as the unique strong solution of a vector-valued stochastic differential equation with a singular drift (\cite{Gra-Mal}, Corollary 6.7.). 

\subsection{Main result}
More generally, following another route inspired by the seminal paper \cite{Bia}, we shall prove the following theorem which is the main result of our paper: 
\begin{teo}\label{Main}
Write $m=(\lambda \theta) d, p = \theta d$. Then for any $ n\geq 1$, there exist real numbers $\left(c_{n,h,l}^{(\lambda, \theta)}\right)_{0 \leq l \leq h-1 \leq n-1}$ such that 
\begin{equation*}
\lim_{d \rightarrow +\infty} \frac{1}{m}\sum_{\substack{\alpha \, \, \textrm{hook} \\ |\alpha| = n, l(\alpha) \leq m}}  \sum_{\substack{\tau \subseteq \alpha \\ \tau \neq \emptyset}}
	\frac{(-1)^{n-\alpha_1} e^{-\nu_\tau(r,s,m) t/d} \, {V}_{\alpha_1, \tau_1}^{r,s,m} \, U_{l(\alpha), l(\tau)}^{r,s,m} }{(r+s+\tau_1+2m-l(\tau)) (\tau_1+l(\tau)-1)} = 
	\sum_{h=1}^n(-1)^{h-1}\frac{e^{-ht}}{h}\sum_{l=0}^{h-1} \frac{(2ht)^{l}}{l!} c_{n,h,l}^{(\lambda,\theta)}.
\end{equation*}
\end{teo}
Since 
\begin{equation*}
\lim_{d \rightarrow \infty} \frac{1}{m} M_n^{(r,s,m)}(\infty), \quad n \geq 1,
\end{equation*}
are the moments of the stationary distribution of the free Jacobi process (see \cite{Dem}, section 5), then Theorem \ref{Main} provides an expression for the moments of the free Jacobi process at any time $t > 0$. As to the description of 
$c_{n,h,l}^{(\lambda, \theta)}$, it is easily read off the end of the proof of Theorem \ref{Main}. Loosely speaking, th proof of this theorem relies on a polynomial long division with respect to the variable $d$ and only the constant term of the quotient polynomial contributes to the large $d$-limit. Besides, this term is shown to be polynomial in two variables, denoted $(j,k)$ in the proof of Theorem \ref{Main}, and $c_{n,h,l}^{(\lambda, \theta)}$ is the coefficient of $j^{h-l-1}k^{n-h}$ there. 

Using the generating series of elementary symmetric and completely homogeneous polynomials, we shall represent  the constant term of the quotient polynomial alluded to above as a doubly symmetric function in two alphabets. 
The latter are indeed the roots of the dividend and of the divisor in the polynomial long division. However, though this representation is interesting in its own, we are not able to use it to obtain a relatively simple expression of $c_{n,h,l}^{(\lambda, \theta)}$. That is why we shall write down a finite and doubly-indexed recurrence equation whose last term is $c_{n,h,l}^{(\lambda, \theta)}$. Doing so leads to simple explicit expressions of both $c_{n,n,l}^{(\lambda, \theta)}$ and $c_{n,h,h-1}^{(\lambda, \theta)}$. 
We also derive the expression 
\begin{equation*}
c_{n,h,h-1}^{(1,\theta)} =	(1-\theta)^n \theta^{n-1} \binom{2n}{n-h},
\end{equation*}
which shows that the binomial coefficient \eqref{Catalan} is not restricted to the value $\theta =1/2$. 

More generally and provided that $\lambda \neq 1, \theta = 1/2$, we realize that the binomial coefficient \eqref{Catalan} is (by the virtue of the Chu-Vandermonde formula) the special value at $\lambda = 1$ of the sum 
\begin{equation*}
\sum_{i=0}^{n-h}\binom{n}{i}\binom{n}{n-h-i}\frac{1}{\lambda^i}, \quad 0 \leq h \leq n.
\end{equation*}
Besides, our doubly-indexed recurrence equation takes a simpler form through binomial coefficients and Gauss hypergeometric functions taken at $\lambda =1$. 
 
The paper is organized as follows. In the next section, we use Carlitz summation formula to give another derivation of the moments of $\mu_t^{(1,1/2)}$. Section three is entirely devoted to the proof of Theorem \ref{Main}. In the fourth section, we prove the representation of the constant term of the quotient polynomial through elementary symmetric and completely homogeneous polynomials. In the fifth section, we derive the expressions of $c_{n,n,l}^{(\lambda, \theta)}$ and $c_{n,h,h-1}^{(\lambda, \theta)}$ as well as the recurrence equation allowing to compute $c_{n,h,l}^{(\lambda, \theta)}$ for arbitrary parameters $(\lambda, \theta)$. We also revisit there the particular case corresponding to $\lambda = 1, \theta=1/2$. In section six, we simplify the recurrence equation as much as possible the recurrence equation for arbitrary $\lambda$ and $\theta = 1/2$.

\section{Carlitz summation formula}
In \cite{DHH}, we proved the following result : 
\begin{teo}[\cite{DHH}, Corollary 1]
The moments of $\mu_t^{(1,1/2)}$ are given by: 
\begin{equation}\label{SpecCas}
\int_0^1 x^n \mu_t^{(1,1/2)}(dx) = \frac{1}{2^{2n}}\binom{2n}{n} + \frac{1}{2^{2n-1}} \sum_{h=1}^n\binom{2n}{n-h} \frac{e^{-ht}}{h}L_{h-1}^{(1)}(2ht), \quad n \geq 1,
\end{equation}
where $L_{h-1}^{(1)}$ is the $(h-1)$-th Laguerre polynomial of index one.
\end{teo}
 In this section, we supply another derivation of \eqref{SpecCas} starting from \eqref{HG} and appealing to Carlitz summation formula (\cite{Abi}). Of course, we shall restrict ourselves to the time-dependent part of the RHS of \eqref{SpecCas} since the stationary one is known to be the $n$-th moment of the arcsine distribution 
\begin{equation*}
\mu_{\infty}^{(1,1/2)} := \frac{1}{\pi\sqrt{x(1-x)}} {\bf 1}_{[0,1]}(x).
\end{equation*} 
\begin{proof}[Another proof of \eqref{SpecCas}]
Substitute $p(m) = m+(1/2), d=2m,$ in \eqref{HG} and recall the Pochhammer symbol: 
\begin{equation*}
(a)_k = \frac{\Gamma(a+k)}{\Gamma(a)}, \, a > 0, \quad (-n)_k = (-1)^k\frac{n!}{(n-k)!}, \, n \in \mathbb{N}.
 \end{equation*}
Then Carlitz Summation formula entails (see e.g. \cite{Abi}, eq. (1)): 
\begin{multline}\label{Carlitz}
{}_4F_3\left(\begin{matrix} -(n-h), m+h-j, m+h-j+1/2, 2m-n+h-2j-1; \\ m-n+h-j, m-n+h-j+1/2, 2m+2h-2j; \end{matrix}, 1\right)  = \\ 
\frac{(m-n+h-j-1/2)}{m-j-1/2}\frac{(-2n)_{n-h}}{(2m-2n+2h-2j-1)_{n-h}} = (-1)^{n-h} \frac{(m-n+h-j-1/2)}{m-j-1/2} \frac{(2n)!}{(n+h)!}\\ \frac{\Gamma(2m-2n+2h-2j-1)}{\Gamma(2m-n+h-2j-1)}. 
 \end{multline}
Taking into account the time change $t \mapsto t/d = t/(2m)$ in \eqref{MomSeq} and keeping in mind \eqref{HG} and \eqref{Eigen1}, we need to compute the limit as $m \rightarrow +\infty$ of: 
\begin{multline*}
\sum_{h=1}^n \frac{e^{-ht}}{h}\binom{2n}{n-h}\sum_{j=0}^{h-1} \frac{(-1)^{j}e^{-h(h-2j-1)t/(2m)}}{(h-j-1)!j!}  \frac{(2m+2h-2j-1)(m-n+h-j-1/2)}{m(2m+h-2j-1)}
\\ \frac{\Gamma(m-j+1/2)\Gamma(m+h-j-1/2)\Gamma(m+h-j)\Gamma(2m-2n+2h-2j-1)\Gamma(2m+h-j)}{(m-j-1/2)\Gamma(m-j-1/2)\Gamma(2m+2h-2j)\Gamma(m+h-n-j+1/2)\Gamma(m+h-n-j)\Gamma(2m-j)}. 
\end{multline*}
To this end, we use Legendre duplication formula to get: 
\begin{eqnarray*}
\frac{\Gamma(m+h-j-1/2)\Gamma(m+h-j)}{\Gamma(2m+2h-2j)}&  = &\frac{\sqrt{\pi}}{2^{2m+2h-2j-1}(m+h-j-1/2)}, \\ 
 \frac{\Gamma(2m-2n+2h-2j-1)}{\Gamma(m+h-n-j+1/2)\Gamma(m+h-n-j)} & = & \frac{2^{2m+2h-2j-2n-2}}{\sqrt{\pi}(m-n+h-j-1/2)}. 
\end{eqnarray*}
Consequently, we are led after some simplifications to compute:  
\begin{multline*}
\lim_{m \rightarrow +\infty}\sum_{h=1}^n \frac{e^{-h(h-1)t/(2m)}}{2^{2n}}\frac{e^{-ht}}{h}\binom{2n}{n-h}\sum_{j=0}^{h-1} \frac{(-1)^{j}e^{jht/m}}{(h-j-1)!j!}  \frac{(2m-2j-1)}{m(2m+h-2j-1)} \frac{\Gamma(2m+h-j)}{\Gamma(2m-j)} = 
 \\ 
\lim_{m \rightarrow +\infty}\sum_{h=1}^n \frac{e^{-h(h-1)t/(2m)}}{2^{2n}}e^{-ht}\binom{2n}{n-h}\sum_{j=0}^{h-1}e^{(2ht)j/(2m)} \binom{h-1}{j} \frac{ (-1)^{j}(2m-2j-1)}{m(2m+h-2j-1)} \frac{(2m+h-j-1)!}{h!(2m-j-1)!}.
\end{multline*}
Up to the factor $(2m-2j-1)/(2m+h-2j-1)$ and to the time change $t \mapsto 2t$, the inner sum
\begin{equation*}
\frac{1}{2m}\sum_{j=0}^{h-1} (-1)^{j}e^{(ht)j/(2m)} \binom{h-1}{j}  \frac{(2m+h-j-1)!}{h!(2m-j-1)!},
\end{equation*}
was considered by Biane in \cite{Bia} (see the proof of Lemma 3 there) who proved the convergence to
\begin{equation*}
\frac{1}{h}L_{h-1}^{(1)}(ht).
\end{equation*}
Writing
\begin{equation*}
\frac{2m-2j-1}{2m+h-2j-1} = 1-\frac{h}{2m+h-2j-1} 
\end{equation*}
we get
\begin{multline*}
\lim_{m \rightarrow \infty} \frac{e^{-h(h-1)t/(2m)}}{2^{2n}}\sum_{h=1}^n e^{-ht}\binom{2n}{n-h}\sum_{j=0}^{h-1} (-1)^{j}e^{(2ht)j/(2m)} \binom{h-1}{j} \frac{2(2m-2j-1)}{(2m)(2m+h-2j-1)} \frac{(2m+h-j-1)!}{h!(2m-j-1)!} \\ 
= \frac{2}{2^{2n}}\sum_{h=1}^n e^{-ht}\binom{2n}{n-h}\frac{1}{h}L_{h-1}^{(1)}(2ht),
\end{multline*}
as desired. 
\end{proof}
A by-product of the previous proof is given in the following corollary: 
\begin{cor}
For any $n \geq 1$, the $n$-th moment of $J_t$ corresponding to the parameter set 
\begin{equation*}
p = m+1/2, \quad d= 2m \quad \Leftrightarrow \quad r= 1/2, \quad s=-1/2,
\end{equation*}
 reads: 
\begin{multline*}
M_{n}^{(1/2,-1/2,m)}(t) = \frac{1}{n!}\sum_{h=0}^{n-1}(-1)^h\binom{n-1}{h}\frac{(m-h)_n(3/2 +m-h-1)_n}{(2m-i)_n} + \\ \sum_{h=1}^n \frac{e^{-h(h-1)t/(2m)}}{2^{2n}}e^{-ht}\binom{2n}{n-h}\sum_{j=0}^{h-1}e^{(2ht)j/(2m)} \binom{h-1}{j} \frac{ (-1)^{j}(2m-2j-1)}{2m+h-2j-1} \frac{(2m+h-j-1)!}{h!(2m-j-1)!}.
\end{multline*}
\end{cor} 
\begin{proof}
The stationary distribution of the eigenvalues of the Hermitian Jacobi process is given by the Selberg weight: 
\begin{equation*}
\prod_{i=1}^m \lambda_i^r(1-\lambda_i)^s \prod_{1 \leq i < v \leq m}(\lambda_i - \lambda_v)^2 {\bf 1}_{\{0 < \lambda_m < \cdots < \lambda_1 < 1\}}, 
\end{equation*}
up to a normalizing constant. Its $n$-th moment was computed in \cite{CDLV}, Corollary II. 3., and our sought expression follows after substituting in that corollary $a=3/2, b= 1/2, N = m, k=n$. 
 \end{proof}

\begin{rem}
In \cite{Chu} (Theorem 2.1), the following reduction formula was proved for terminating $1$-balanced ${}_4F_3$ at unity: 
\begin{multline*}
{}_4F_3\left(\begin{matrix} -N, a-c+N, (c/2), (c+1)/2; \\ 1+a-e, (e/2), (e+1)/2; \end{matrix}, 1\right) = \frac{(1+a-c-e)_N(e-c)_N}{(1+a-e)_N(e)_N}{}_3F_2\left(\begin{matrix} -N, a-c+N, c; \\ c+e-a-N, e+N; \end{matrix}, 1\right).
\end{multline*}
Specializing it with
\begin{equation*}
N = n-h, \quad p = p(m) = m + \frac{1}{2}, \quad c = 2(m+h-j), \quad e = 2(m+h-j-n), 
\end{equation*}
we get
\begin{multline*}
{}_4F_3\left(\begin{matrix} -(n-h), m+h-j, m+h-j+1/2, d-n+h-2j-1; \\ m-n+h-j, m-n+h-j+1/2, d+2h-2j; \end{matrix}, 1\right)  = \frac{(d-2m)_{n-h}(-2n)_{n-h}}{(d+2h-2j)_{n-h}(2m+2h-2j-2n)_{n-h}}\\ 
{}_3F_2\left(\begin{matrix} -(n-h),  d-n+h-2j-1, 2m+2h-2j; \\ 2m - d +h+1-n, 2m-2j+h-n; \end{matrix}, 1\right). 
 \end{multline*}
The RHS of the last equation may be further written as (\cite{Chu}, Corollary 2.2): 
\begin{equation*}
\frac{(-2n)_{n-h}}{(2m+2h-2j-2n)_{n-h}}{}_3F_2\left(\begin{matrix} -(n-h),  2m-d+1, 2m+2h-2j; \\ 2m-2j+h-n, d+2h-2j; \end{matrix}, 1\right), 
\end{equation*}
which reduces when $d=2m$ to Cariltz summation formula \eqref{Carlitz}:
\begin{equation*}
\frac{(-2n)_{n-h}}{(2m+2h-2j-2n)_{n-h}}{}_2F_1\left(\begin{matrix} -(n-h), 1, \\ 2m-2j+h-n; \end{matrix}, 1\right) = \frac{(m-n+h-j-1/2)(-2n)_{n-h}}{(m-j-1/2)(2m-2n+2h-2j-1)_{n-h}}, 
\end{equation*}
by the virtue of the Gauss hypergeometric Theorem: 
\begin{equation*}
{}_2F_1(-N, b;c, 1) = \frac{(c-b)_{N}}{(c)_N}. 
\end{equation*}
However, we do not know whether a simple expression of the above terminating ${}_3F_2$ hypergeometric function exists for $d \neq 2m$. 
\end{rem}

\section{The general case: proof of Theorem \ref{Main}} 
We now proceed with the general case and prove Theorem \ref{Main}. The proof is long and for ease of reading, we divide it into three steps: 
\begin{enumerate}
\item Express the summands of the second term in the RHS of \eqref{FinalForm} as a rational function in the variable $d$. 
\item Inspired by the proof of Lemma 2 in \cite{Bia}, we expand the exponential factor $e^{(2thj)/d}$ into powers of $1/d$, truncate the obtained series and perform there a polynomial long division. 
\item Obtain cancellations and determine exactly the terms having non zero contributions in the large $d$-limit. The expression displayed in Theorem \ref{Main} follows after straightforward computations. 
\end{enumerate}
\subsection{First step} For any $n \geq 1$, denote
\begin{align*}
	\tilde{M}_n^{(r,s,m)}(t)  := \frac{1}{m} \sum_{\substack{\alpha \, \, \textrm{hook} \\ |\alpha| = n, l(\alpha) \leq m}} (-1)^{n-\alpha_1}  \sum_{\substack{\tau \subseteq \alpha \\ \tau \neq \emptyset}}
	\frac{e^{-\nu_\tau(r,s,m) t/d} \, V_{\alpha_1, \tau_1}^{(r,s,m)} \, U_{l(\alpha), l(\tau)}^{(r,s,m)} }{(r+s+\tau_1+2m-l(\tau)) (\tau_1+l(\tau)-1)},
\end{align*}
where we recall that $\nu_\tau(r,s,m)$, $V_{\alpha_1, \tau_1}^{(r,s,m)}$ and $U_{l(\alpha), l(\tau)}^{(r,s,m)}$ were defined in \eqref{Eigen}, \eqref{V} and \eqref{U}  respectively. 
Recall also the parameterization of the hooks $\tau \subseteq \alpha$: 
\begin{equation*}
\alpha_1=n-k, \, \tau_1=h-j, \quad l(\alpha)=k+1, \, l(\tau)=j+1. 
\end{equation*}
Consequently, we may rewrite $\tilde{M}_n^{(r,s,m)}(t)$ as
\begin{align*}
	\tilde{M}_n^{(r,s,m)}(t) &= \frac{1}{m} \sum_{h=1}^n\sum_{j=0}^{h-1}\sum_{k=j}^{j+n-h}(-1)^kF(m,p,d,n,h,j,k)e^{-ht-\frac{ht}{d}(h-(2j+1))}
\end{align*}
where 
\begin{multline*}
	F(m,p,d,n,h,j,k)= \frac{\Gamma(d+h-j)\Gamma(n-k+m)\Gamma(d-j-k-1)\Gamma(p+n-k)\Gamma(q+h-j)}{(n-k-h+j)!(k-j)!j!\Gamma(d+n-k+h-j)\Gamma(p+h-j)\Gamma(h-j)}
	\\ \frac{(d+2h-2j-1)(d-2j-1)\Gamma(p-j)}{h(d+h-2j-1)\Gamma(m-k)\Gamma(p-k)\Gamma(q-j)\Gamma(d-j)}.
\end{multline*}
Performing the index changes $k\mapsto n-h+j-k$, we further obtain
\begin{align}\label{For1}
\tilde{M}_n^{(r,s,m)}(t) &=  \sum_{h=1}^n\frac{e^{-ht - th(h-1)/d}}{n!}\binom{n}{h}\sum_{k=0}^{n-h}\sum_{j=0}^{h-1}(-1)^{n-h-k+j}\binom{n-h}{k}\binom{h-1}{j}\tilde{F}(m,p,d,n,h,j,k)e^{\frac{2thj}{d}}
\end{align}
where
\begin{align*}
	\tilde{F}(m,p,d,n,h,j,k)=&\frac{(d-2j-1)(d+2h-2j-1)(d-j)_h(d-p-j)_h(m+h-n-j+k)_n(p+h-n-j+k)_n}{m(d+h-2j-1)(p-j)_h(d-n+h-2j-1+k)_{n+h+1}}.
\end{align*}
Write $m=(\lambda \theta) d, p = \theta d$, then 
\begin{align*}
	\tilde{F}(m,p,& d,n,h,j,k) = \tilde{F}((\lambda \theta)d,\theta d,d,n,h,j,k)
	\\& =\frac{(d-2j-1)(d+2h-2j-1)(d-j)_h((1-\theta)d-j)_h((\lambda \theta)d+h-n-j+k)_n(\theta d+h-n-j+k)_n}{(\lambda \theta)d(d+h-2j-1)(\theta d-j)_h(d-n+h-2j-1+k)_{n+h+1}}.
\end{align*}
\subsection{Second step}
Since the Pochhammer symbol $(a)_k$ has exactly $k$ factors, then a quick inspection shows that 
\begin{align}\label{AB}
	\tilde{F}((\lambda \theta)d,\theta d,d,n,h,j,k) \sim \frac{(1-\theta)^h(\lambda \theta^2)^n}{(\lambda \theta)\theta^h} d^{n-1}, \quad d = d(m) \rightarrow +\infty,
\end{align}
so that taking the limit as $m \rightarrow +\infty$ in \eqref{For1} yields an indeterminate form unless $n = 1$. Note that the situation was similar in \cite{Bia} when dealing with the moments of the unitary Brownian motion. 
In that paper, the author succeeded to obtain cancellations making use of the following elementary, yet useful, lemma: 
\begin{lem}[\cite{Bia}, Lemma 2]
Let $0 \leq a \leq b$ be integers. Then 
\begin{equation*}
\sum_{j=0}^b (-1)^j j^b \binom{b}{j} = \left\{\begin{array}{lcr} 
0 & \textrm{if} & a < b, \\ 
(-1)^a a! & \textrm{if} & a =b.
\end{array}
\right.
\end{equation*}
\label{LemBiane}
\end{lem} 
Mimicking the proof of Lemma 3 in \cite{Bia}, we expand the exponential term $e^{\frac{2thj}{d}}$ in \eqref{For1}: 
\begin{align*}
	\sum_{k=0}^{n-h}(-1)^{n-h-k} \sum_{j=0}^{h-1}(-1)^{j}\binom{h-1}{j}\binom{n-h}{k}\tilde{F}((\lambda \theta) d,\theta d,d,n,h,j,k)\left(\sum_{l=0}^{n-1}\frac{(2thj)^l}{l!d^l}+O(d^{-n})\right).
\end{align*}
By the virtue of \eqref{AB}, we only need to analyse on the following triple sum:
\begin{align*}
\sum_{l=0}^{n-1}\frac{(-2th)^l}{l!}\sum_{k=0}^{n-h}\sum_{j=0}^{h-1} (-1)^{n-h-k+j}\binom{h-1}{j}j^l\binom{n-h}{k}\frac{\tilde{F}((\lambda \theta) d,\theta d,d,n,h,j,k)}{d^l}.
\end{align*}
More precisely, we shall use Lemma \ref{LemBiane} to obtain cancellations in the $(j,k)$-double sum. To this end, let\footnote{For ease of notations, we shall omit the parameters $(\lambda,\theta)$.}
\begin{equation*}
P_{n,h}(j,k,d):=(d-2j-1)(d+2h-2j-1)(d-j)_h((1-\theta)d-j)_h((\lambda \theta) d+h-n-j+k)_n(\theta d +h-n-j+k)_n
\end{equation*}
denote the numerator of $\tilde{F}((\lambda \theta) d,\theta d,d,n,h,j,k)/d^l$, and for fixed $0 \leq l \leq n-1$, let
\begin{equation*}
D_{n,h,l}(j,k,d):= (\lambda \theta) d^{l+1}(d+h-2j-1)(\theta d-j)_h(d-n+h-2j-1+k)_{n+h+1},
\end{equation*}
be its denominator. Performing polynomial long division with respect to the variable $d$ for fixed $(n,h,l)$, there exists a polynomial (quotient) $Q_{n,h,l}(j,k,d)$ of degree $n-l-1$ such that 
\begin{equation*}
P_{n,h}(j,k,d) = D_{n,h,l}(j,k,d)Q_{n,h,l}(j,k,d) + \, \textrm{Remainder}.
\end{equation*}
Consequently, the asymptotic behaviour of $\tilde{M}_n^{(r,s,m)}(t)$ as $d \rightarrow +\infty$ is governed by:
 \begin{align}\label{sum3}
 	\sum_{l=0}^{n-1} \frac{(2th)^l}{l!} \sum_{j=0}^{h-1}(-1)^{j}\binom{h-1}{j}j^l\sum_{k=0}^{n-h}(-1)^{n-h-k}\binom{n-h}{k}Q_{n,h,l}(j,k,d).
 \end{align}
\subsection{Third step}
Write
\begin{equation}\label{Q-Exp}
Q_{n,h,l}(j,k,d) := \sum_{i=0}^{n-l-1}q_{n-l-1-i}^{(n,h)}(j,k)d^i, 
\end{equation}
and note that its coefficients $(q_{i}^{(n,h)}(j,k))_{i=0}^{n-l-1}$ do not depend on the index $l$ since the latter only appears in the factor $d^l$ of $D_{n,h,l}(j,k,d)$. 
The following lemma describes the polynomial structure of these coefficients in the variables $(j,k)$. 
\begin{lem}\label{degree}
For any $0\le i\le n-l-1$, the coefficient $q_i^{(n,h)}(j,k)$ is a polynomial of degree at most $i$ in both variables $j$ and $k$. Moreover, the coefficient of $j^s$ there (resp. $k^s$) is a polynomial of degree at most $i-s$ in $k$ (resp. $j$). 
\end{lem}

\begin{proof}
Let $(a_i^{(n,h)})_{0\le i\le 2n+2h+2}$ and $(b_i^{(n,h)})_{0\le i\le n+2h+l+3}$ be respectively the coefficients of the polynomials $P_{n,h}(j,k,\cdot)$ and $D_{n,h,l}(j,k,\cdot)$ so that:
\begin{equation*}
P_{n,h}(j,k,d) =\sum_{i=0}^{2n+2h+2}a_{i}^{(n,h)}(j,k)d^i \quad {\rm and\ } D_{n,h,l}(j,k,d) =\sum_{i=l+1}^{n+2h+l+3}b^{(n,h)}_{i}(j,k)d^i.
\end{equation*}
Since the degree of the remainder of the polynomial long division is strictly less than $B$, then $(q_{i}^{(n,h)})_{i=0}^{n-l-1}$ satisfy the following relations:
\begin{equation}\label{q-coef}
 q_{i}^{(n,h)}(j,k) = \tilde{a}_{2n+2h+2-i}^{(n,h)}(j,k)  - \sum_{v=0}^{i-1}\tilde{b}_{n+2h+l+3-i+v}^{(n,h)}(j,k)q_{v}^{(n,h)}(j,k), \quad  0 \leq i \leq n-l-1
\end{equation}
where
\begin{equation*}
	\tilde{a}_{2n+2h+2-i}^{(n,h)}(j,k)=\frac{a_{2n+2h+2-i}^{(n,h)}(j,k)}{(\lambda\theta)\theta^h}, \quad \tilde{b}_{n+2h+l+3-i+v}^{(n,h)}(j,k)=\frac{b_{n+2h+l+3-i+v}^{(n,h)}(j,k)}{(\lambda\theta)\theta^h}.
\end{equation*}
Consequently, we readily infer from \eqref{q-coef} by induction and from the product forms of $P_{n,h}(j,k,d)$ and $D_{n,h,l}(j,k,d)$ that $q_{i}^{(n,h)}(j,k)$ is a polynomial of degree at most $i$ in both $j$ and $k$. 
The last statements follows similarly  from \eqref{q-coef} by induction and a close (yet straightforward) inspection of the coefficients of $j^s$ and of $k^s$ in $q_{i}^{(n,h)}(j,k)$. 
\end{proof}

\begin{rem}
Later, we shall express the coefficients $(q_{i}^{(n,h)}(j,k))$ through elementary and completely homogeneous symmetric polynomials in the roots of $P_{n,h}(j,k,d)$ and of $D_{n,h}(j,k,d)$ respectively. These expressions may be used to write another proof of the previous lemma. 
 \end{rem}
With the help of Lemma \ref{LemBiane} and Lemma \ref{degree}, we are now able to obtain all the cancellations occurring in \eqref{sum3} and determine the terms having non zero contributions in the large $d$-limit:
\begin{pro}
	All the terms in \eqref{sum3} vanish except those corresponding to the $(k,j)$-monomials $k^{n-h}j^{h-1-l}$ in the constant term coefficient $q_{n-l-1}^{(n,h)}(j,k)$ of the polynomial $Q_{n,h,l}(j,k,d)$.
\end{pro}
\begin{proof}

By the virtue of \eqref{sum3} and of \eqref{Q-Exp}, we need firstly to show that
\begin{align*}
 \sum_{j=0}^{h-1}(-1)^{j}\binom{h-1}{j}j^l\sum_{k=0}^{n-h}(-1)^{n-h-k}\binom{n-h}{k}q_{n-l-1-i}^{(n,h)}(j,k)
\end{align*}
is zero for any $1\le i\le n-l-1$. According to Lemma \ref{LemBiane} and since $q_{n-l-1-i}^{(n,h)}$ is a polynomial of degree at most $n-l-1-i$ in $j$ and $k$ by Lemma \ref{degree}, we notice that the terms corresponding to $l > h-1$ in \eqref{sum3} give zero contribution. Indeed, 
\begin{align*}
	n-l-1-i=n-(l+1)-i < n-h-i < n-h
\end{align*}
so that the inner sum over $k$ is zero for all $1\le i\le n-l-1$.

Now consider a $j$-monomial of degree $s$ in $q_{n-l-1-i}^{(n,h)}$. If $s+l < h-1$ then the sum over $j$ vanishes. Otherwise $s+l\ge h-1$ and the sum over $k$ is zero for any $1\le i\le n-l-1$ since
\begin{align*}
	n-l-1-i-s=n-1-i-(l+s)\le n-1-i-(h-1)=n-h-i < n-h.
\end{align*}
Secondly, let $i=0$ and recall from lemma \ref{degree} that the coefficient of $j^s$ in $q_{n-l-1}^{(n,h)}$ is a polynomial in $k$ of degree at most $n-l-1-s$. Reasoning as above, if $s+l < h-1$ then the sum over $j$ vanishes while the sum over $k$ does so if $s+l > h-1$ since 
\begin{align*}
	n-l-1-s=n-1-(l+s) < n-1-(h-1)=n-h.
\end{align*}
Hence, we are only left with the monomials $k^{v}j^{h-1-l}, v \leq n-h,$ in $q_{n-l-1}^{(n,h)}(j,k)$. But once more, the sum over $k$ vanishes unless $v=n-h$, the proposition is proved.   
\end{proof}

The end of the proof of Theorem \ref{Main} goes as follows. The sum \eqref{sum3} reduces to 
\begin{align*}\label{sum4}
	\sum_{l=0}^{h-1} \frac{(2th)^l}{l!} \sum_{j=0}^{h-1}(-1)^{j}\binom{h-1}{j}j^l\sum_{k=0}^{n-h}(-1)^{n-h-k}\binom{n-h}{k}c_{n,h,l}^{(\lambda,\theta)}k^{n-h}j^{h-l-1}
\end{align*}
where $c_{n,h,l}^{(\lambda,\theta)}$ is the coefficient of the monomial $k^{n-h}j^{h-l-1}$ in $q_{n-l-1}^{(n,h)}(j,k)$. Appealing again to Lemma \ref{LemBiane}, we may simplify further this triple sum to get
\begin{equation*}
\sum_{l=0}^{h-1} \frac{(2th)^l}{l!} (-1)^{h-1}(h-1)!(n-h)!c_{n,h,l}^{(\lambda,\theta)},
\end{equation*}
so that
\begin{align*}
\lim_{d \rightarrow \infty}\tilde{M}_n^{(r,s,m)}(t) &= \sum_{h=1}^n\frac{e^{-ht}}{n!}\binom{n}{h}(-1)^{h-1}(n-h)!(h-1)!\sum_{l=0}^{h-1} \frac{(2th)^l}{l!} c_{n,h,l}^{(\lambda,\theta)} 
\\ & =  \sum_{h=1}^n(-1)^{h-1}\frac{e^{-ht}}{h}\sum_{l=0}^{h-1} \frac{(2ht)^{l}}{l!} c_{n,h,l}^{(\lambda,\theta)},
\end{align*}
as stated in Theorem \ref{Main}. 

\section{Representations of the constant term coefficient $q_{n-l-1}(j,k)$} 
So far, we proved that the large size asymptotics of the moments of $J_t$ are governed by the constant term coefficient $q_{n-l-1}(j,k)$ of the polynomial $Q_{n,h,l}(j,k,\cdot)$. Besides, this coefficient is the last term $i=n-l-1$ of the inductive relation \eqref{q-coef}. 
In this section, we derive another representation of $q_{n-l-1}(j,k)$ and of all the coefficients of $Q_{n,h,l}(j,k,\cdot)$ through elementary and complete homogeneous symmetric polynomials taken at the roots of $P_{n,h}(j,k,\cdot)$ and of $D_{n,h}(j,k,\cdot)$. 
Though we believe that this fact is very likely known, we were not able to find any reference to it and as such, we state and prove it below. Using a result of Goulden and Greene (\cite{Gou-Gre}), the obtained representation may be further written as a complete symmetric polynomial in some set of variables (alphabet) formed out of the the roots of $P_{n,h}(j,k,\cdot)$ and of $D_{n,h}(j,k,\cdot)$. 

For sake of simplicity, set 
\begin{equation*}
A := 2n+2h+2, \quad B := n+2h+l+3, \quad C := A-B = n-l-1, 
\end{equation*}
for fixed $n,h,l,$ and denote $(y_1,\dots,y_A)$ and $(z_1, \dots, z_B)$  the roots of the polynomials $P_{n,h}(j,k,\cdot)$ and $D_{n,h,l}(j,k,\cdot)$ respectively (which depend of course of $(n,h,j,k)$).  
Then, 
\begin{pro}
The polynomial $Q_{n,h,l}(j,k, d)$ admits the following expansion:  
\begin{align}\label{Desc0}
Q_{n,h,l}(j,k,d) &= d^{A-B}\frac{a_A}{b_B}\sum_{i=0}^{A-B} \frac{(-1)^i}{d^i}\sum_{v=0}^i(-1)^ve_{i-v}(y_1,\dots, y_A) h_{v}(z_1,\dots, z_B) \nonumber
\\& = \frac{a_A}{b_B}\sum_{i=0}^{C}d^i (-1)^{C-i}\sum_{v=0}^{C-i}(-1)^v e_{C-i-v}(y_1,\dots, y_A) h_{v}(z_1,\dots, z_B),
\end{align}
where $e_i(y_1,\dots, y_A)$ and $h_i(z_1,\dots, z_B)$ are the $i$-th elementary and complete homogeneous symmetric polynomials respectively. 
In particular, the constant term coefficient of $Q_{n,h,l}(j,k,\cdot)$ is given by 
\begin{equation}\label{Desc1}
q_C^{(n,h)}(j,k) = (-1)^C \frac{(1-\theta)^h(\lambda \theta^2)^n}{\theta^h(\lambda\theta)} \sum_{v=0}^C (-1)^v e_{C-v}(y_1,\dots, y_A) h_v(z_1,\dots, z_B).
\end{equation}
\end{pro} 
\begin{proof}
 One obviously has:
\begin{equation*}
P_{n,h}(j,k,d) = a_Ad^A \prod_{v=1}^A(1-y_v/d) = a^Ad^A\sum_{v=0}^{\infty}(-1)^ve_{v}(y_1,\dots, y_A)\frac{1}{d^v},
\end{equation*}
where we set $e_{v}(y_1,\dots, y_A) = 0$ for any $v > A$, and (for large enough $d$):
\begin{equation*}
\frac{1}{D_{n,h}(j,k,d)} = \frac{1}{b_Bd^{B}}\prod_{v=1}^B(1-z_v/d) = \frac{1}{b_Bd^{B}}\sum_{v=0}^{+\infty} h_{v}(z_1,\dots, z_B)\frac{1}{d^v}.
\end{equation*}
As a result, 
\begin{equation*}
\frac{P_{n,h}(j,k,d)}{D_{n,h}(j,k,d)} = d^{A-B}\frac{a_A}{b_B}\sum_{i=0}^{+\infty} \frac{(-1)^i}{d^i}\sum_{v=0}^i (-1)^ve_{i-v}(y_1,\dots, y_A) h_{v}(z_1,\dots, z_B).
\end{equation*}
But the quotient $Q_{n,h,l}(j,k,d)$ of the polynomial division in the LHS is the polynomial part of the RHS in the variable $d$ whence the first expression of $Q_{n,h,l}(j,k,d)$ follows. Performing the index change $i \mapsto A-B-i=C-i$, one gets the second one. 

Finally, the leading coefficients of $P_{n,h}(j,k,d)$ and $D_{n,h}(j,k,d)$ are given by
\begin{equation*}
a_A = \theta^n (1-\theta)^h(\lambda \theta)^n, \quad b_B = (\lambda \theta) \theta^h,
\end{equation*}
respectively. Extracting the constant term coefficient in \eqref{Desc0}, the lemma is proved. 
 \end{proof}

\begin{rem}
Recalling the representation \eqref{Q-Exp}
\begin{equation*}
Q_{n,h,l}(j,k,d) = \sum_{i=0}^{n-l-1}q_{n-l-1-i}^{(n,h)}(j,k)d^i = \sum_{i=0}^{C}q_{C-i}^{(n,h)}(j,k)d^i,  
\end{equation*}
we readily infer from \eqref{Desc0} that 
\begin{equation*}
q_{C-i}^{(n,h)}(j,k)= (-1)^{C-i}\sum_{v=0}^{C-i}(-1)^ve_{C-i-v}(y_1,\dots, y_A) h_{v}(z_1,\dots, z_B), \quad 0 \leq i \leq C,
\end{equation*}
On the other hand, the roots $(y_1, \dots, y_A)$ are given by: 
\begin{equation}\label{root1}
y_{i}=
\begin{cases}
2j-2h+1, & i=1\\
2j+1, & i=2\\
j-i+3, & 3\le i\le h+2\\
(j+h-i+3)/(1-\theta), & h+3\le i\le 2h+2\\
(j-k+n+h-i+3)/(\lambda \theta), & 2h+3\le i\le n+2h+2\\
(j-k+2n+h-i+3)/\theta, & n+2h+3\le i\le 2n+2h+2=A
\end{cases},
\end{equation}
while 
\begin{equation}\label{root2}
z_{i}=
\begin{cases}
0, & 1\le i\le l+1\\
2j-h+1, & i=l+2\\
(j-i+l+3)/\theta, & l+3\le i\le h+l+2\\
2j-k+n+l-i+4, & h+l+3\le i\le n+2h+l+3 = B
\end{cases}.
\end{equation}
As a matter fact, the simple monomial $j$ is present in all the roots except $z_i, 1\le i\le l+1$. Since $h_v$ and $e_{C-i-v}$ are homogenous polynomials (in their variables) of orders $v$ and $C-i-v$ respectively, then \eqref{Desc0} shows that $q_{C-i}^{(n,h)}(j,k)$ 
is a polynomial of degree at most $C-i$ in the variable $j$. As to the simple monomial $k$, it is present in more than $C$ roots. Consequently, the same conclusion holds regarding the degree of $q_{C-i}^{(n,h)}(j,k)$ in the variable $k$ and we get another proof of Lemma \ref{degree}.
\end{rem} 

The formula \eqref{Desc1} is clearly symmetric in each of the alphabets $\mathcal{Y}:= (y_1, \dots, y_A)$ and $\mathcal{Z}:= (z_1, \dots, z_B)$. In the realm of algebraic combinatorics, a function enjoying this property is referred to as doubly symmetric. Moreover, $q_C$ may be seen as a completely homogeneous function in an another alphabet formed out of $\mathcal{Y}$ and $\mathcal{Z}$. Indeed, appealing again to the homogeneity of $h_i$, we write
\begin{align*}
q_C &= (-1)^C \frac{(1-\theta)^h(\lambda \theta^2)^n}{\theta^h(\lambda\theta)} \sum_{i=0}^C h_i(-z_1,\dots, -z_B) e_{C-i}(y_1,\dots, y_A)
\\& = (-1)^C \frac{(1-\theta)^h(\lambda \theta^2)^n}{\theta^h(\lambda\theta)} \sum_{i=0}^C h_i(-z_1,\dots, -z_B, \underbrace{0, \dots, 0}_{C \textrm{times}}) e_{C-i}(y_1,\dots, y_A).
\end{align*}
Moreover, Lemma 2.1 and formula (18) in \cite{Gou-Gre} entail:
\begin{cor}
The constant term coefficient $q_C$ admits the following representation:
\begin{equation*}
q_C^{(n,h)}(j,k)  = (-1)^C \frac{(1-\theta)^h(\lambda \theta^2)^n}{\theta^h(\lambda\theta)} \sum_{1 \leq \tau_1 \leq \tau_2 \leq \dots \leq \tau_C \leq A} \prod_{i=1}^C(y_{\tau_i + i-1} - z_{\tau_i}),
\end{equation*}
where it is understood that $y_i = 0, i > A$ and $z_i = 0, i > B$. 
\end{cor}

\section{On the coefficient $c_{n,h,l}^{(\lambda,\theta)}$}
The limiting moment formula stated in Theorem \ref{Main} depends only on the coefficient $c_{n,h,l}^{(\lambda,\theta)}$ of the monomial $k^{n-h}j^{h-l-1}$ in the $(j,k)$-polynomial $q_{n-l-1}^{(n,h)}(j,k)$. The extraction of this coefficient may obtained from the above representations of $q_{n-l-1}^{(n,h)}(j,k)$ by successive differentiations with respect to $(j,k)$ (viewed as real variables). In this section, we shall rather use \eqref{q-coef} to compute inductively  $c_{n,h,l}^{(\lambda,\theta)}$ for any 
$(\lambda, \theta)$. As a warm up, we firstly investigate the two special (univariate) cases corresponding to $n=h$ and to $l=h-1$, for which the coefficients of $j^{n-l-1}$ and of $k^{n-l-1}$ in $q_{n-l-1}^{(n,h)}(j,k)$ are constant (by the virtue of the second statement of Lemma \ref{degree}). Doing so reveals further the role played by the special values $\lambda = 1$ and $\theta = 1/2$. For instance, we determine explicitly $c_{n,h,h-1}^{(\lambda,\theta)}$ whose expression considerably simplifies when 
$\lambda =1$ without any restriction on $\theta$. The same holds for $c_{n,n,l}^{(\lambda,1/2)}$ while $c_{n,n,l}^{(\lambda,\theta)}$ is expressed through complete ordinary Bell polynomials.

\subsection{The case $n=h$}
For any $0 \leq i \leq n-l-1$, let $X_i^{(n,\lambda,\theta)}$ denote the coefficient of $j^i$ in $q_{i}^{(n,n)}(j,k)$, so that 
\begin{equation*}
c_{n,n,l}^{(\lambda,\theta)}= X_{n-l-1}^{(n,\lambda,\theta)}.
\end{equation*} 
Then 
\begin{pro}\label{case n=h}
	We have, $X_0^{(n,\lambda,\theta)}=(\lambda\theta)^{n-1}(1-\theta)^n$ and for any $1\le i \le n-l-1$, $X_i^{(n,\lambda,\theta)}$ satisfies
	\begin{align}\label{n=h}
		X_i^{(n,\lambda,\theta)}=&\alpha_i	-\sum _{v=0}^{i-1} \beta_{i-v}^{(n,\theta)}  X_v^{(n,\lambda,\theta)}
	\end{align}
where
\begin{equation*}
	\alpha_i^{(n,\lambda,\theta)} = (-1)^i\sum _{i_1+i_2+i_3+i_4+i_5=i}   \binom{2}{i_1}  \binom{n}{i_2} \binom{n}{i_3} \binom{n}{i_4}  \binom{n}{i_5}2^{i_1}(\lambda\theta)^{n-1-i_3} (1-\theta)^{n-i_4}  \theta^{-i_5}
\end{equation*}
and
\begin{equation*}
	\beta_{i-v}^{(n,\theta)}=(-1)^{i-v} \sum _{i_1+i_2+i_3=i-v} \binom{1}{i_1} \binom{n}{i_2} \binom{2 n+1}{i_3} 2^{i_1+i_3} \theta^{-i_2}, \quad 0 \leq v \leq i-1.
\end{equation*}
\end{pro}

\begin{proof}
From Lemma \ref{degree}, $X_0^{(n,\lambda,\theta)}$ is the highest degree coefficient $q_0^{(n,h)}$ in $Q_{n,h,l}(j.k,d)$ and we infer from \eqref{q-coef} together with the definition of $P_{n,h,l}(j, k,d)$ that:
\begin{equation*}
 q_{0}^{(n,h)}(j,k) = \tilde{a}_{2n+2h+2}^{(n,h)}(j,k) = (\lambda\theta)^{n-1}\theta^{n-h}(1-\theta)^n
 \end{equation*}
so that  $X_0^{(n,\lambda,\theta)} = q_{0}^{(n,h)}(j,k) = (\lambda\theta)^{n-1}(1-\theta)^n$. 

Appealing again to \eqref{q-coef}, we get \eqref{n=h} where $\alpha_i^{(n,\lambda,\theta)}$ and $\beta_{i-v}^{(n,\lambda,\theta)}, 0 \leq v \leq i-1,$ are the coefficients of $j^i$ and of $j^{i-v}$ in 
$\tilde{a}_{2n+2h+l+2-i}^{(n,h)}(j,k)$ and in $\tilde{b}_{n+2h+l+3-i+v}^{(n,h)}(j,k)$ respectively.
Indeed, the first statement of Lemma \ref{degree} shows that $q_v^{(n,h)}(j,k), 1\le v \le i-1,$ is a polynomial of degree at most $v$ in the variable $j$. Moreover,
	\begin{align}\label{Sym1}
		\tilde{a}_{2n+2h+2-i}^{(n,h)}(j,k) &= (-1)^i(\lambda\theta)^{n-1}\theta^{n-h}(1-\theta)^h\sum_{1\le j_1< \dots  j_i \le 2n+2h+2}y_{j_1}\ldots y_{j_i}, \quad 1 \leq i \leq n-l-1,
	\end{align}
	and
	\begin{equation}\label{Sym2}
		\tilde{b}_{n+2h+l+3-i+v}^{(n,h)}(j,k)=(-1)^{i-v} \sum_{1\le j_1 < \dots < j_{i-v}\le n+2h+l+3}z_{j_1}\ldots z_{j_{i-v}},  \quad 0 \leq v \leq i-1,
	\end{equation}
are elementary symmetric polynomials in the roots $(y_1,\dots, y_{2n+2h+1})$ and $(z_1, \dots, z_{n+2h+l+3})$ respectively. As such, they are polynomials of degrees $i$ and $i-v$ in each of the variables $j$ and $k$. 

Finally, the expressions of $\alpha_i$ and of $\beta_{i-v}$ readily follow from \eqref{root1} and \eqref{root2}, keeping in mind $n=h$. The proposition is proved.
\end{proof}

\subsection{The case $\lambda =1, \theta =1/2$.}
If $\theta =1/2$ then the generalized Chu-Vandermonde identity entails:
\begin{align*}
\beta_{i-v}^{(n,1/2)} & =(-1)^{i-v} \sum _{i_1+i_2+i_3=i-v} \binom{1}{i_1} \binom{n}{i_2} \binom{2 n+1}{i_3} 2^{i_1+i_2+i_3} 
\\& = (-2)^{i-v} \sum _{i_1+i_2+i_3=i-v} \binom{1}{i_1} \binom{n}{i_2} \binom{2 n+1}{i_3}
\\& =  (-2)^{i-v}\binom{3n+2}{i-v}.
\end{align*}
Since $\beta_0=1$ then the recurrence relation \eqref{n=h} may be then written as: 
\begin{equation*}
\sum _{v=0}^{i}  (-2)^{i-v}\binom{3n+2}{i-v} X_v^{(n,\lambda,1/2)} = \alpha_i^{(n,\lambda,1/2)} ,
\end{equation*}
and is inverted as: 
\begin{equation*}
X_i^{(n,\lambda,1/2)} = \sum_{v=0}^i \frac{(3n+2)_{i-v}}{(i-v)!}\alpha_i^{(n,\lambda,1/2)} , \quad 0 \leq i \leq n-l-1.
\end{equation*}
In particular, 
\begin{equation*}
X_{n-l-1}^{(n,\lambda,1/2)}= \sum_{v=0}^{n-l-1} \frac{(3n+2)_{n-l-1-v}}{(n-l-1-v)!}\alpha_v^{(n,\lambda,\theta)}.
\end{equation*}
If further $\lambda = 1$ then 
\begin{align*}
\alpha_i^{(n,1,1/2)} &= \frac{(-1)^i}{2^{2n-1}}\sum _{i_1+i_2+i_3+i_4+i_5=i}   \binom{2}{i_1}  \binom{n}{i_2} \binom{n}{i_3} \binom{n}{i_4}  \binom{n}{i_5}2^{i_1+i_3+i_4+i_5}
\\& = \frac{(-1)^i}{2^{2n-1}}\sum_{i_2=0}^i\binom{n}{i_2} 2^{i-i_2} \sum _{i_1+i_3+i_4+i_5=i-i_2}   \binom{2}{i_1}   \binom{n}{i_3} \binom{n}{i_4}  \binom{n}{i_5}
\\& = \frac{(-1)^i}{2^{2n-1}}\sum_{i_2=0}^i\binom{n}{i_2} \binom{3n+2}{i-i_2} 2^{i-i_2} 
\end{align*}
where the last equality follows again from the generalized Chu-Vandermonde identity. In this case, \eqref{n=h} simplifies further to
\begin{equation*}
\sum _{v=0}^{i}  (-1)^{v}\binom{3n+2}{i-v} 2^{i-v}X_v^{(n,1,1/2)}  = \frac{1}{2^{2n-1}}\sum_{v=0}^i\binom{n}{v} \binom{3n+2}{i-i_2} 2^{i-v}. 
\end{equation*}
and the uniqueness of its solution (subject to the initial condition $X_0= 1/2^{2n-1}$) clearly yields
\begin{equation*}
X_v^{(n,1,1/2)} = \frac{(-1)^v}{2^{2n-1}} \binom{n}{v}. 
\end{equation*}
As a result, 
\begin{equation*}
c_{n,n,l}^{(1,1/2)}= X_{n-l-1}^{(n,1,1/2)}  = \frac{(-1)^{n-l-1}}{2^{2n-1}} \binom{n}{n-l-1} = \frac{(-1)^{n-l-1}}{2^{2n-1}} \binom{n}{l+1},
\end{equation*}
whence we deduce that the $h=n$ term of the sum displayed in the RHS of Theorem \ref{Main} reads 
\begin{equation*}
\frac{1}{2^{2n-1}}\frac{e^{-nt}}{n}\sum_{l=0}^{h-1} \frac{(-2nt)^{l}}{l!}\binom{n}{l+1} = \frac{e^{-nt}}{2^{2n-1}n}L_{n-1}^{(1)}(2nt),
\end{equation*}
which agrees with \eqref{SpecCas}. More generally,

\begin{pro}\label{Toeplitz}
	For all $0\le l\le n-1$, 	we have
	\begin{equation*}
		c_{n,n,l}^{(\lambda,\theta)}=\sum_{v=0}^{n-l-1}\alpha_v^{(n,\lambda,\theta)} \gamma_{n-l-1-v}^{(n,\theta)}
	\end{equation*} 
	where
	\begin{equation*}
		\gamma_a^{(n,\theta)} =\sum _{\substack{v_1 \geq 0, \dots, v_a \geq 0 \\ v_1+2s_2\ldots+av_a=a}} \frac{(v_1+\cdots + v_a)!}{v_1!\ldots v_a!}(-\beta_1^{(n,\theta)})^{v_1}\ldots (-\beta_a^{(n,\theta)})^{v_a}, \quad 0 \leq a \leq n-l-1.
	\end{equation*}
\end{pro}
\begin{proof}
The system of equations \eqref{n=h} may be written as 
	\begin{equation*}
		\left(\begin{array}{c}
			\alpha_0^{(n,\lambda,\theta)} \\ 
			\alpha_1^{(n,\lambda,\theta)} \\
			\vdots \\
			\alpha_{n-l-1}^{(n,\lambda,\theta)}
		\end{array}\right) = 
		\underbrace{\left(\begin{array}{cccc}
				1 & 0 & \dots & 0 \\ 
				\beta_{1}^{(n,\theta)} & 1 & \dots & 0  \\ 
				\vdots & \vdots & \ddots & \vdots \\
				\beta_{n-l-1}^{(n,\theta)} &\dots & \dots & 1 
			\end{array}
			\right)}_{\mathcal{B}}\left(\begin{array}{c}
			X_0^{(n,\lambda,\theta)} \\ 
			X_1^{(n,\lambda,\theta)} \\
			\vdots \\
			X_{n-l-1}^{(n,\lambda,\theta)}
		\end{array}\right)
	\end{equation*}
	where $\mathcal{B}$ is a lower triangular invertible Toeplitz matrix. Appealing to Corollary 1 in \cite{Mer}, we obtain the expression of the inverse matrix $\mathcal{B}^{-1}$ whence the formula for $c_{n,n,l}^{(\lambda,\theta)}$ follows. 
	\end{proof}

\begin{rem}
For any $1 \leq a \leq n-l-1$, $\gamma_a^{(n,\theta)}$ is a complete ordinary Bell polynomial in the variables $(-\beta_1^{(n,\theta)}, \dots, -\beta_a^{(n,\theta)})$. 
\end{rem}

\subsection{The case $l=h-1$}
This case seems a bit easier than the previous one and the proof of the following proposition is similar to that of Proposition \ref{n=h}.
\begin{pro}\label{case l=h-1}
	For all $1\le h\le n$, we have
	\begin{equation*}
		c_{n,h,h-1}^{(\lambda,\theta)}= (1-\theta)^h(\lambda\theta)^{n-1}(-\theta)^{n-h}\binom{2n}{n-h}\sum_{v=0}^{n-h} \frac{(h-n)_v(-n)_v}{(-2n)_v v!\theta^{v}}{}_2F_1(-v,-n,n-v+1; 1/\lambda),
	\end{equation*}
	where ${}_2F_1$ is the Gauss hypergeometric function. In particular, for any $n\ge1$, we have
	\begin{equation*}
		c_{n,h,h-1}^{(1,\theta)} =	(1-\theta)^n \theta^{n-1} \binom{2n}{n-h}.
	\end{equation*}
\end{pro}
\begin{proof}
	Let $Y_i^{(n, h, \lambda, \theta)}$ denote the coefficient $k^i$ in $q_{i}^{(n,h)}(j,k)$ so that $c_{n,h,h-1}^{(\lambda,\theta)}= Y_{n-h}^{(n, h, \lambda, \theta)}$. Then 
	\begin{equation*}
	Y_0=(\lambda\theta)^{n-1}\theta^{n-h}(1-\theta)^h
	\end{equation*}
	and \eqref{q-coef} shows that the sequence $(Y_v^{(n, h, \lambda, \theta)})_{0 \leq v \leq i}$ satisfies
	\begin{equation}\label{l=h-1}
		Y_i^{(n, h, \lambda, \theta)} = \alpha_i^{(n, h, \lambda, \theta)}-\sum _{v=0}^{i-1}  Y_v^{(n, h, \lambda, \theta)}\beta_{i-v}^{(n, h)}, \quad  0\le i\le n-h.
	\end{equation}
where (by the virtue of \eqref{root1}, \eqref{root2}, \eqref{Sym1} and \eqref{Sym2}):
	\begin{align*}
		\alpha_i^{(n, h, \lambda, \theta)} &= (1-\theta)^h\sum _{v=0}^i \binom{n}{v}  \binom{n}{i-v}(\lambda\theta)^{n-1-v} \theta^{n-h-(i-v)} 
		\\& = (1-\theta)^h(\lambda\theta)^{n-1}\theta^{n-h-i}\sum _{v=0}^i \binom{n}{v}  \binom{n}{i-v}(\lambda)^{-v}
	\end{align*}
	and
	\begin{equation*}
		\beta_i^{(n, h)} = \binom{h+n+1}{i}.
	\end{equation*}
The recurrence relation \eqref{l=h-1} may be then written as: 
\begin{equation*}
\sum_{v=0}^i \binom{h+n+1}{i-v}Y_v = (1-\theta)^h(\lambda\theta)^{n-1}\theta^{n-h-i}\sum _{v=0}^i \binom{n}{v}  \binom{n}{i-v}(\lambda)^{-v},
\end{equation*}
and is inverted as follows: for any $0 \leq i \leq n-h$, one has:
\begin{equation*}
Y_{i}^{(n, h, \lambda, \theta)} = (1-\theta)^h(\lambda\theta)^{n-1}\theta^{n-h}\sum_{v=0}^i (-1)^{i-v}\frac{(n+h+1)_{i-v}}{(i-v)!} \theta^{-v} \sum _{a=0}^v \binom{n}{a}  \binom{n}{v-a}(\lambda)^{-a}.
\end{equation*}
Consequently, 
\begin{align*}
Y_{n-h}^{(n, h, \lambda, \theta)} &= \frac{(1-\theta)^h(\lambda\theta)^{n-1}\theta^{n-h}}{(n+h)!}\sum_{v=0}^{n-h} (-1)^{n-h-v}\frac{(2n-v)!}{(n-h-v)!} \theta^{-v} \sum _{a=0}^v \binom{n}{a}  \binom{n}{v-a}(\lambda)^{-a} 
\\& = (1-\theta)^h(\lambda\theta)^{n-1}(-\theta)^{n-h}\binom{2n}{n-h}\sum_{v=0}^{n-h} (-1)^v\frac{(h-n)_v}{(-2n)_v} \frac{\theta^{-v}n!}{v!} \sum _{a=0}^v\binom{n}{a} \frac{v!}{(v-a)!(n-(v-a))!}  (\lambda)^{-a} 
\\& = (1-\theta)^h(\lambda\theta)^{n-1}(-\theta)^{n-h}\binom{2n}{n-h}\sum_{v=0}^{n-h} \frac{(h-n)_v}{(-2n)_v} \frac{\theta^{-v}(-n)_v}{v!}\sum _{a=0}^v\frac{(-n)_a(-v)_a}{a!(n-v+1)_a}  (\lambda)^{-a} 
\\& =  (1-\theta)^h(\lambda\theta)^{n-1}(-\theta)^{n-h}\binom{2n}{n-h}\sum_{v=0}^{n-h} \frac{(h-n)_v(-n)_v}{(-2n)_v v!\theta^{v}} {}_2F_1(-v,-n,n-v+1; 1/\lambda).
 \end{align*}
Finally, if $\lambda = 1$ then 
\begin{equation*}
{}_2F_1(-v,-n,n-v+1; 1) = \frac{(2n-v+1)_v}{(n-v+1)_v} = \frac{(2n!)(n-v)!}{n!(2n-v)!},
\end{equation*}
by the Gauss hypergeometric Theorem, whence 
\begin{align*}
Y_{n-h}^{(n, h, 1, \theta)} &=   (1-\theta)^h(\theta)^{n-1}(-\theta)^{n-h}\binom{2n}{n-h}\sum_{v=0}^{n-h} \frac{(h-n)_v}{v!\theta^{v}} 
\\& =  (1-\theta)^h(\theta)^{n-1}(-\theta)^{n-h}\binom{2n}{n-h}\left(1-\frac{1}{\theta}\right)^{n-h}
\\& = (1-\theta)^n \theta^{n-1} \binom{2n}{n-h}.
\end{align*}
 \end{proof}

\subsection{The general case}
Denote  $Z_{a,b}^{(n, h, \lambda, \theta)}$ the coefficient of $j^ak^b$ in $q_{a+b}^{(n,h)}(j,k)$. Recall also that $c_{n,h,l}^{(\lambda,\theta)}$ is the coefficient of $j^{h-l-1}k^{n-h}$ in 
\begin{equation*}
q_{n-l-1}^{(n,h)}(j,k) =q_{(h-l-1)+(n-h)}^{(n,h)}(j,k),
\end{equation*}
in other words,
\begin{equation*}
	c_{n,h,l}^{(\lambda,\theta)}=Z_{h-l-1,n-h}^{(n, h, \lambda, \theta)}.
\end{equation*}
Appealing once more to \eqref{q-coef}, we prove the following:
\begin{pro}\label{ProGS}
$Z_{0,0}^{(n, h, \lambda, \theta)} = (\lambda\theta)^{n-1}\theta^{n-h}(1-\theta)^h$ and for any $0\le a\le h-l-1$ and  $0\le b\le n-h, (a,b) \neq (0,0)$, one has:
	\begin{equation}\label{General-sys}
		Z_{a,b}^{(n, h, \lambda, \theta)}=A_{a,b}^{(n, h, \lambda, \theta)} -\sum_{\substack{0\le u\le a,\ 0\le v\le b\\0\le u+v\le a+b-1}} Z_{u,v}\  B_{a-u,b-v}^{(n, h, \theta)}
	\end{equation}
where
\begin{multline*}
	A_{a,b}^{(n, h, \lambda, \theta)} =(-1)^a\sum _{i=0}^b \binom{n}{i} \binom{n}{b-i}\\ \sum _{i_1+i_2+i_3+i_4+i_5=a}  \binom{2}{i_1}  \binom{h}{i_2} \binom{h}{i_3}\binom{n-i}{i_4}\binom{n -(b-i)}{i_5}2^{i_1} (1-\theta )^{h-i_3}  (\theta  \lambda )^{n-1-i-i_4}  \theta ^{(n-h) - (b-i)-i_5},
\end{multline*}
and
\begin{equation*}
	B_{a,b}^{(n, h, \theta)}=(-1)^a\binom{n+h+1}{b} \sum _{j_1+j_2+j_3=a} \binom{1}{j_1}\binom{h}{j_2}\binom{n+h+1-b}{j_3} \frac{2^{j_1+j_3}}{\theta^{j_2}}. 
\end{equation*}
\end{pro}

\begin{proof}
We follow the lines of the proof of Propositions \ref{case n=h} and \eqref{l=h-1}: $Z_{0,0}^{(n, h, \lambda, \theta)}$ is the highest degree coefficient $q_0^{(n,h)}(j,k)$ in $Q_{n,h,l}(j.k,d)$ and its expression follows readily from \eqref{Sym1}. As to
$A_{a,b}^{(n, h, \lambda, \theta)}$ and $B_{a,b}^{(n, h, \theta)}$ are nothing else but the coefficients of $j^a k^b$  in $\tilde{a}_{2n+2h+2-a-b}^{(n,h)}$ and $\tilde{b}_{n+2h+l+3-a-b}^{(n,h)}$ respectively, viewed as bivariate polynomials in $(j,k)$. In this respect, 
note from \eqref{Sym1} and \eqref{Sym2} that their total degrees do not exceed $a+b$ since the roots $(y_1, \dots, y_{2n+2h+2})$ and $(z_1, \dots, z_{n+2h+l+3})$ are linear factors in $(j,k)$. The double recurrence relation \eqref{General-sys} follows then from \eqref{q-coef} and it only remains to derive the expressions of $A_{a,b}^{(n, h, \lambda, \theta)}$ and $B_{a,b}^{(n, h, \theta)}$. 

To derive the former, we consider the $2n= n +n$ roots 
\begin{equation*}
y_{i}=
\begin{cases}
(j-k+n+h-i+3)/(\lambda \theta), & 2h+3\le i\le n+2h+2\\
(j-k+2n+h-i+3)/\theta, & n+2h+3\le i\le 2n+2h+2
\end{cases},
\end{equation*}
from which we choose $b$ elements. Obviously, there are
\begin{equation*}
\binom{n}{i} \binom{n}{b-i}, 0 \leq i \leq b,
\end{equation*}
possible choices, each contributing with a factor $(-1)^bk^b/[(\lambda \theta)^i\theta^{b-i}]$. Once a choice is fixed, we choose further $a$ elements from the remaining $2n+2h+2-b$ roots in \eqref{root1}. The contribution of any choice is of the form 
\begin{equation*}
\frac{2^{i_1} 1^{i_2}}{(1-\theta)^{i_3}(\lambda \theta)^{i_4}\theta^{i_5}}
\end{equation*}
where $i_1+i_2+i_3+i_4+i_5 = a$ and $0 \leq i_1 \leq 2$ and there are
\begin{equation*}
\binom{2}{i_1}  \binom{h}{i_2} \binom{h}{i_3}\binom{n-i}{i_4}\binom{n-(b-i)}{i_5}
 \end{equation*}
possible choices. Keeping in mind \eqref{Sym1}, the expression of $A_{a,b}^{(n, h, \lambda, \theta)}$ follows. The derivation of the expression of $B_{a,b}^{(n, h, \theta)}$ is similar, we are done. 
\end{proof}

\begin{rem}[Back to the case $\lambda =1, \theta =1/2$]
When $\lambda=1, \theta = 1/2$, the expressions of $A_{a,b}^{(n, h, 1, 1/2)}$ and of $B_{a,b}^{(n, h, 1/2)}$ reduce to: 
\begin{eqnarray*}
A_{a,b}^{(n, h, 1, 1/2)} & = & \frac{(-1)^a2^b}{2^{2n-1}} \binom{2n}{b}\sum_{i=0}^a2^{a-i}\binom{h}{i}\binom{2n+h+2-b}{a-i}, \\ 
B_{a,b}^{(n, h, 1/2)} &=& (-1)^a 2^a \binom{n+h+1}{b}\binom{n+2h+2-b}{a}.
\end{eqnarray*}
Consequently, \eqref{General-sys} reads: 
\begin{multline}\label{GenSysBis}
\sum_{\substack{0\le u\le a,\ 0\le v\le b\\ u+v\le a+b}} 2^{a-u}(-1)^u\binom{n+h+1}{b-v}\binom{n+2h+2-b+v}{a-u} Z_{u,v}^{(n, h, 1,1/2)}  \\ =  \frac{2^b}{2^{2n-1}}\binom{2n}{b}\sum_{i=0}^{a}2^{a-i}\binom{h}{i}\binom{2n+h+2-b}{a-i}.
\end{multline}
On the other hand, comparing both \eqref{SpecCas} and Theorem \ref{Main}, and using the expansion of the Laguerre polynomial: 
\begin{equation*}
L_{h-1}^{(1)}(2ht) = \sum_{l=0}^{h-1}\frac{(-2ht)^l}{l!}\binom{h}{l+1} = \sum_{l=0}^{h-1}\frac{(-2ht)^l}{l!}\binom{h}{h-l-1},
\end{equation*}
it follows that 
\begin{equation*}
Z_{h-l-1,n-h}^{(n, h, 1,1/2)} = \frac{(-1)^{h-l-1}}{2^{2n-1}} \binom{h}{h-l-1}\binom{2n}{n-h}.
\end{equation*}
However, the general expression of $Z_{u,v}^{(n, h, 1,1/2)}$ for arbitrary $(u,v)$ seems out of reach for the moment (even with the help of computer-assisted computations) though it considerably simplifies when $u=h-l-1, v = n-h$. 
\end{rem}
 
\section{further perspectives: $\theta = 1/2$} 
The expressions of $A_{a,b}^{(n, h, \lambda, \theta)}$ and of $B_{a,b}^{(n, h, \theta)}$ in Proposition \ref{ProGS} suggest considering the value $\theta=1/2$ for arbitrary $\lambda$. In this respect, we still have:   
\begin{equation*}
B_{a,b}^{(n, h, 1/2)} = (-1)^a 2^a \binom{n+h+1}{b}\binom{n+2h+2-b}{a}, 
\end{equation*}
while $A_{a,b}^{(n, h, \lambda, 1/2)}$ reduces to: 
\begin{multline*}
A_{a,b}^{(n, h, \lambda, 1/2)} =\frac{(-1)^a \lambda^{n-1}2^b}{2^{2n-1}}\sum _{i=0}^b \binom{n}{i} \binom{n}{b-i} \frac{1}{\lambda^i}
\sum _{i_1+i_2+i_3+i_4+i_5=a}  \binom{2}{i_1}  \binom{h}{i_2} \binom{h}{i_3}\binom{n-i}{i_4}\binom{n -(b-i)}{i_5}\frac{2^{a-i_2}}{\lambda^{i_4}}.
\end{multline*}
We can write this last expression as follows: 
\begin{pro}
For any $0 \leq a \leq h-l-1, 0 \leq b \leq n-h$,  
\begin{multline*}
A_{a,b}^{(n, h, \lambda, 1/2)} =\frac{(-1)^a \lambda^{n-1}2^b}{2^{2n-1}}\sum_{i_2=0}^{a}\frac{2^{a-i_2}}{\lambda^{a-i_2}} \binom{h}{i_2} \binom{2n+h+2-b}{a-i_2}
\\ \sum _{i=0}^b \binom{n}{i} \binom{n}{b-i} \frac{1}{\lambda^i}{}_2F_1\left(i_2-a, b-i-n-h-2, b-2n-h-2; 1-\lambda\right).
\end{multline*}
\end{pro}
\begin{proof}
We split the inner sum as 
\begin{equation*}
\sum_{i_2+i_4 = 0}^a\binom{h}{i_2} \binom{n-i}{i_4} \frac{2^{a-i_2}}{\lambda^{i_4}} \sum_{i_1+i_3+i_5 = a-i_2-i_4}  \binom{2}{i_1} \binom{h}{i_3}\binom{n -(b-i)}{i_5}
\end{equation*}
and use the Chu-Vandermonde identity to get: 
\begin{multline*}
A_{a,b}^{(n, h, \lambda, 1/2)} =\frac{(-1)^a \lambda^{n-1}2^b}{2^{2n-1}}\sum _{i=0}^b \binom{n}{i} \binom{n}{b-i} \frac{1}{\lambda^i}
\sum_{i_2+i_4 = 0}^a\binom{h}{i_2} \binom{n-i}{i_4} \binom{n+h+2-(b-i)}{a-i_2-i_4}\frac{2^{a-i_2}}{\lambda^{i_4}}.
\end{multline*}
Next we focus on the sum 
\begin{equation*}
\sum_{i_4 = 0}^{a-i_2}\binom{n-i}{i_4} \binom{n+h+2-(b-i)}{a-i_2-i_4}\frac{1}{\lambda^{i_4}}
\end{equation*}
for fixed $0 \leq i_2 \leq a, 0 \leq i \leq b$. Using the identity, 
\begin{equation*}
\binom{N}{j} = (-1)^j \frac{(-N)_j}{j!}, \quad 0 \leq j \leq N, 
\end{equation*}
we get 
\begin{align*}
\sum_{i_4 = 0}^{a-i_2}\binom{n-i}{i_4} \binom{n+h+2-(b-i)}{a-i_2-i_4}\frac{1}{\lambda^{i_4}} = \binom{n+h+2-(b-i)}{a-i_2}{}_2F_1\left(i_2-a, i-n, n+h+3-(b-i)+i_2-a; \frac{1}{\lambda}\right).
\end{align*}
Appealing to the argument transformation for terminating Gauss hypergeometric functions (see e.g. \cite{Vid}, eq. 22): 
\begin{equation*}
{}_2F_1(-N, b,c; z) = \frac{(c-b)_N}{(c)_N} z^N{}_2F_1 \left(-N, 1-c-N, b-c-N+1; 1-\frac{1}{z}\right),
\end{equation*}
we can further write 
\begin{multline*}
\sum_{i_4 = 0}^{a-i_2}\binom{n-i}{i_4} \binom{n+h+2-(b-i)}{a-i_2-i_4}\frac{1}{\lambda^{i_4}} = \binom{2n+h+2-b}{a-i_2}\frac{1}{\lambda^{a-i_2}} 
\\ {}_2F_1\left(i_2-a, b-i-n-h-2, b-2n-h-2; 1-\lambda\right)
\end{multline*}
As a result, 
\begin{multline*}
A_{a,b}^{(n, h, \lambda, 1/2)} =\frac{(-1)^a \lambda^{n-1}2^b}{2^{2n-1}}\sum _{i=0}^b \binom{n}{i} \binom{n}{b-i} \frac{1}{\lambda^i}
\\ \sum_{i_2=0}^{a}\frac{2^{a-i_2}}{\lambda^{a-i_2}} \binom{h}{i_2} \binom{2n+h+2-b}{a-i_2}{}_2F_1\left(i_2-a, b-i-n-h-2, b-2n-h-2; 1-\lambda\right)
\end{multline*}
whence the proposition follows after inverting the order of summation. 
\end{proof}
By the virtue of this proposition, the two-parameters relation \eqref{General-sys}
\begin{equation*}
Z_{a,b}^{(n, h, \lambda, \theta)}=A_{a,b}^{(n, h, \lambda, \theta)} -\sum_{\substack{0\le u\le a,\ 0\le v\le b\\0\le u+v\le a+b-1}} Z_{u,v}^{(n, h, \lambda, \theta)}  B_{a-u,b-v}^{(n, h, \theta)},
\end{equation*}
is given explicitly by
\begin{multline}\label{theta=1/2}
\sum_{u=0}^a 2^{a-u}(-1)^u\sum_{v=0}^{b}\binom{n+h+1}{b-v}\binom{n+2h+2-b+v}{a-u} Z_{u,v}^{(n, h, \lambda,1/2)} =  \frac{\lambda^{n-1}2^b}{2^{2n-1}}
\sum_{u=0}^{a}\frac{2^{a-u}}{\lambda^{a-u}} \binom{h}{u} \binom{2n+h+2-b}{a-u}
\\ \sum _{i=0}^b \binom{n}{i} \binom{n}{b-i} \frac{1}{\lambda^i}{}_2F_1\left(u-a, b-i-n-h-2, b-2n-h-2; 1-\lambda\right).
\end{multline}

\end{document}